\documentclass[10pt]{amsart}
\usepackage{amssymb}
\usepackage{color}
\usepackage{graphicx}
\usepackage{float}
\usepackage[all,cmtip]{xy}
\newcommand{\SF}{{\mathcal{F}}}
\newcommand{\SL}{{\mathcal{L}}}
\newcommand{\im}{{\operatorname{Image}}}
\newcommand{\Id}{{\operatorname{Id}}}
\newcommand{\wind}{{\operatorname{wind}}}
\newcommand{\ind}{{\operatorname{ind}}}
\newcommand{\dist}{{\operatorname{dist}}}
\newcommand{\area}{{\operatorname{area}}}
\newcommand{\ev}{{\operatorname{ev}}}

\newcommand{\End}{{\operatorname{End}}}
\newcommand{\R}{{\mathbb{R}}}

\newtheorem{lemma}{Lemma}
\newtheorem{proposition}[lemma]{Proposition}
\newtheorem{theorem}[lemma]{Theorem}

\newtheorem{definition}[lemma]{Definition}

\theoremstyle{remark}
\newtheorem{remark}[lemma]{Remark}
\newtheorem{example}[lemma]{Example}

\begin{document} 

\title{The Foliated Weinstein Conjecture}

\date{May, 2015}


\author{\'Alvaro del Pino}
\address{Universidad Aut\'onoma de Madrid and Instituto de Ciencias Matem\'aticas -- CSIC.
C. Nicol\'as Cabrera, 13--15, 28049, Madrid, Spain.}
\email{alvaro.delpino@icmat.es}

\author{Francisco Presas}
\address{Instituto de Ciencias Matem\'aticas -- CSIC.
C. Nicol\'as Cabrera, 13--15, 28049, Madrid, Spain.}
\email{fpresas@icmat.es}

\begin{abstract}
A foliation is said to admit a foliated contact structure if there is a codimension $1$ distribution in the tangent space of the foliation such that the restriction to any leaf is contact. We prove a version of the Weinstein conjecture in the presence of an overtwisted leaf. The result is shown to be sharp.
\end{abstract}

\maketitle

\section{Introduction}

The \textbf{Weinstein conjecture} \cite{Wei} states that the Reeb vector field associated to a contact form $\alpha$ in a closed $(2n+1)$--manifold $M$ always carries a closed periodic orbit. Hofer proved in \cite{Ho} that the Weinstein conjecture holds for any 3-dimensional contact manifold $(M^3, \alpha)$ overtwisted or satisfying $\pi_2(M) \neq 0$. Then, it was proven in full generality by Taubes \cite{Tau} by localising the Seiberg-Witten equations along Reeb orbits.

The main theorem of this note -- definitions of the relevant objects will be given in the next section -- reads as follows:

\begin{theorem} \label{thm:main}
Let $(M^{3+m}, \SF^3, \xi^2)$ be a contact foliation in a closed manifold $M$. Let $\alpha$ be a defining 1--form for an extension of $\xi$ and let $R$ be its Reeb vector field. Let $\SL^3 \hookrightarrow M$ be a leaf.
\begin{itemize}
\item[i.] If $(\SL, \xi|_{\SL})$ is an overtwisted contact manifold, $R$ possesses a closed orbit in the closure of $\SL$.
\item[ii.] If $\pi_2(\SL) \neq 0$, $R$ possesses a closed orbit in the closure of $\SL$.
\end{itemize}
\end{theorem}

The case where the leaf $\SL$ is closed corresponds to the Weinstein conjecture. This result constrasts, just as in the non--foliated case, with the behaviour of \textsl{smooth flows}: it was proven in \cite{CPP15} that any never vanishing vector field tangent to a foliation $(M^{3+m}, \SF^3)$ can be homotoped, using parametric plugs, to a tangent vector field without periodic orbits. 

The proof of Theorem \ref{thm:main}, based on Hofer's methods, occupies the last section of the note. Before that, several examples showing the sharpness of the result are discussed. 

In Subsection \ref{ssec:tight}, Proposition \ref{prop:noGeodesics} constructs a contact foliation in the 4--torus $\mathbb{T}^4$ that has \textsl{all leaves tight} and that has no Reeb orbits. Naturally, in this example all leaves are open. This shows that the \textbf{foliated Weinstein conjecture} does not necessarily hold as soon as we drop the assumption on overtwistedness. Then Proposition \ref{prop:geodesic} presents a more sophisticated example of a contact foliation in $\mathbb{S}^3 \times \mathbb{S}^1$ with all leaves tight and with closed Reeb orbits appearing only in the unique compact leaf of the foliation. 

In Subsection \ref{ssec:sharp} we construct a foliation in $\mathbb{S}^2 \times \mathbb{S}^1 \times \mathbb{S}^1$ that has two compact leaves $\mathbb{S}^2 \times \mathbb{S}^1 \times \{0,\pi\}$ on which all others accumulate. We then endow it with a foliated contact structure that makes all leaves overtwisted but that has closed Reeb orbits only in the compact ones. Theorem \ref{thm:main} is therefore sharp in the sense that an overtwisted leaf might not possess a Reeb orbit \textsl{itself}.

In Subsection \ref{ssec:nonComplete} we construct Reeb flows with no closed orbits in every \textsl{open contact manifold}. 

\textsl{Acknowledgements.} Part of this work was carried out while the first author was visiting the group of Andr\'as Stipsicz at the Alfr\'ed R\'enyi Institute. We are very thankful to Klaus Niederkr\"uger for the many hours of insightful conversations about this project. The authors would also like to thank Viktor Ginzburg for his interest in this work and for his suggestions regarding Propositions \ref{prop:openLeaf} and \ref{prop:noGeodesics}. Several other people have discussed with us when this work was in progress: S. Behrens, R. Casals, M. Hutchings, E. Miranda, D. Pancholi, and D. Peralta--Salas.

The first author is supported by a La Caixa--Severo Ochoa grant. Both authors are supported by Spanish National Research Project MTM2013--42135 and the ICMAT Severo Ochoa grant SEV--2011--0087 through the V. Ginzburg Lab.

\section{The relevant concepts involved}

All objects considered henceforth will be smooth. Foliations and distributions will be oriented and cooriented. Often, arguments where orientability assumptions are dropped would go through by taking double or quadruple covers appropriately. 

\subsection{Contact structures}

\begin{definition}
Let $W$ be a $2n+1$ dimensional manifold. A distribution $\xi^{2n} \subset TW$ is said to be a \textbf{contact distribution} if it is maximally non--integrable. A $1$--form $\alpha \in \Omega^1(W)$ satisfying $\ker(\alpha) = \xi$ is called a \textbf{contact form}. $\xi$ being maximally non--integrable amounts to $\alpha$ satisfying $\alpha \wedge d\alpha^n \neq 0$. 

We say that the pair $(W,\xi)$ is a \textbf{contact manifold}.
\end{definition}

A map $\phi: (W_1, \xi_1) \to (W_2,\xi_2)$ satisfying $\phi^* \xi_2 = \xi_1$ is a contact map. If $\phi$ is additionally a diffeomorphism we will say that $\phi$ is a \textbf{contactomorphism}.

\begin{example}
Consider $\mathbb{R}^{2n+1}$ with coordinates $(x_1,y_1,\cdots,x_n,y_n,z)$. The $2n$--distribution $\xi_{st} = \ker(dz-\sum_{i=1..n} x_idy_i)$ is called the standard \textbf{tight} contact structure.
\end{example}

\begin{example}
Consider $\mathbb{R}^3$ with cylindrical coordinates $(r,\theta,z)$. The $2$--distribution $\xi_{ot} = \ker(\cos(\theta)dz+r\sin(r)d\theta)$ is called the contact structure \textbf{overtwisted at infinity}. The disc $\Delta = \{z=0, r \leq \pi\}$ is called the \textbf{overtwisted disc}.
\end{example}

It was shown by Bennequin in \cite{Be} that the structures $(\mathbb{R}^3, \xi_{st})$ and $(\mathbb{R}^3, \xi_{ot})$, although homotopic as plane fields, are distinct as contact structures. 

\subsubsection{Overtwisted contact structures in dimension $3$}

\begin{definition}
Let $(W^3,\xi^2)$ be a contact manifold. $(W,\xi)$ is said to be an \textbf{overtwisted} contact manifold if there is an embedded 2--disc $D \subset W$ and a contactomorphism $\phi: \nu(\Delta) \to \nu(D)$ between a neighbourhood $\nu(\Delta)$ of the overtwisted disc $\Delta \subset \mathbb{R}^3$ and a neighbourhood $\nu(D) \subset W$ of $D$. 
\end{definition}

The relevance of this notion stems from the following theorem stating that overtwisted contact manifolds are completely classified by their underlying algebraic topology.

\begin{theorem} \emph{(Eliashberg \cite{El89})} \label{thm:El89}
Let $W^3$ be a $3$--fold. Any plane field $\eta \subset TW$ is homotopic to an overtwisted contact structure. 

Further, any two overtwisted contact structures $\xi_1, \xi_2 \subset TW$ homotopic as plane fields are homotopic through overtwisted contact structures. In particular, they are contactomorphic.
\end{theorem}

Theorem \ref{thm:El89} says that 2--plane fields and contact structures in $3$--manifolds present a 1 to 1 correspondence at the level of connected components. Eliashberg's result is stronger than what we have stated. Indeed, there is a weak homotopy equivalence if one restricts to the class of plane fields that have a fixed overtwisted disc. 

Overtwisted contact structures in $\mathbb{R}^3$ were completely classified by Eliashberg in \cite{El93}. In particular, the following proposition will be used in Subsection \ref{ssec:sharp}.

\begin{proposition} \emph{(Eliashberg \cite{El93})} \label{prop:El93}
Let $\xi$ be a contact structure in $\mathbb{R}^3$ that is overtwisted in the complement of every compact subset. Then $\xi$ is isotopic to $\xi_{ot}$. 
\end{proposition}

Contact structures with the property that they remain overtwisted after removing any compact subset are called \textbf{overtwisted at infinity}.

\subsubsection{Overtwisted contact structures in higher dimensions}

Overtwisted contact structures have been defined in full generality -- for every dimension -- in \cite{BEM}. In \cite{CMP} it has been shown that the overtwisted disc in higher dimensions can be understood as an stabilisation of the overtwisted disc in dimension $3$.

The following lemma will be useful in Subsection \ref{ssec:nonComplete}. Its proof is based on a \textsl{swindling} argument, as found in \cite{El92}.

\begin{lemma} \label{coro:BEM} \emph{(\cite[Corollary 1.4]{BEM})}
Let $(M^{2n+1}, \xi_M)$ be a connected overtwisted contact manifold and let $(N^{2n+1}, \xi_N)$ be an open contact manifold of the same dimension. Let $f : N \rightarrow M$ be a smooth embedding covered by a contact bundle homomorphism $\Phi : TN \rightarrow TM$ -- that is, $\Phi|_{\xi_M(p)}$ maps into $\xi_N(f(p))$ and preserves the conformal symplectic structure -- and assume that $df$ and $\Phi$ are homotopic as injective bundle homomorphisms $TN \rightarrow TM$.

Then $f$ is isotopic to a contact embedding $\tilde{f} : (N, \xi_N) \rightarrow (M, \xi_M)$.
\end{lemma}

\subsubsection{Convex surfaces} \label{sssec:convex}

Let $(W^3, \xi^2)$ be a contact manifold. Let $\Sigma^2 \subset W$ be an immersed surface. The intersection $\xi \cap T\Sigma$ yields a singular foliation by lines on $\Sigma$, which is called the \textbf{characteristic foliation}. In the generic case, it can be assumed that the singularities -- the points where $\xi_p = T_p\Sigma$ -- are isolated points, that can then be classified into \textbf{nicely elliptic} and \textbf{hyperbolic}.

\begin{example} \label{ex:ot}
By our characterisation of overtwistedness, any overtwisted manifold $(W, \xi)$ contains a disc $\Sigma$ with a single singular point, which is nicely elliptic and whose boundary is legendrian. All other leaves spiral around the legendrian boundary in one end and converge to the elliptic point in the other. Such a disk appears as a $C^\infty$--small perturbation of the overtwisted disk $\Delta$.
\end{example}

\begin{example} \label{ex:tight}
Consider the unit sphere $\mathbb{S}^2$ in $(\mathbb{R}^3, \xi_{st})$. Its singular foliation has two critical points located in the poles, which are nicely elliptic. All other leaves are diffeomorphic to $\mathbb{R}$ and they connect the poles. 
\end{example}

\begin{theorem}[Eliashberg, Giroux, Fuchs] \label{thm:EGF}
Let $\Sigma = \mathbb{S}^2$ and let $(W, \xi)$ be tight. Then, after a $C^0$--small perturbation of its embedding, it can be assumed that the characteristic foliation of $\Sigma$ is conjugate to the one of the unit sphere in $\mathbb{R}^3$ tight. 
\end{theorem}

\subsection{Contact foliations}

The contents of this section appear in more detail in \cite{CPP}.

\begin{definition}
A \textbf{contact foliation} is a triple $(M^{2n+1+m}, \SF^{2n+1}, \xi^{2n})$ where $M$ is a manifold of dimension $2n+1+m$, $\SF$ is a foliation of codimension $m$, and $\xi \subset T\SF$ is a distribution of dimension $2n$ that is contact on each leaf of $\SF$. 

Often we will say that $\xi$ is a \textbf{foliated contact structure} on the foliation $(M, \SF)$. 
\end{definition}

Contact foliations do exist in abundance as the following result shows:
\begin{theorem} $($ \cite{CPP} $)$
Let $(M^{3+m}, \SF^3)$ be a foliation such that the structure group of $T\SF$ reduces to $U(1)\oplus 1$. Then $\SF$ admits a foliated contact structure with all leaves overtwisted.
\end{theorem}
This result is the foliated counterpart of Eliashberg's result \cite{El89}. 

We say that a distribution $\Theta^{2n+m}$ satisfying $\xi = \Theta \cap T\SF$ is an \textbf{extension} of $\xi$, and a regular equation $\alpha$ can be considered for $\Theta = \ker(\alpha)$. It follows that $d\alpha$ is a symplectic form on $\xi$, but not necessarily on $\Theta$. 

\begin{definition}
Let $(M,\SF,\xi)$ be a contact foliation. Let $\Theta$ be an extension of $\xi$ with regular equation $\alpha$. The \textbf{Reeb vector field} $R$ associated to $\alpha$ is the unique vector field satisfying $R \in \Gamma(T\SF)$, $(i_R d\alpha)|_{T\SF} = 0$, and $\alpha(R) = 1$. 
\end{definition}

Of course this is nothing but the leafwise Reeb vector field induced by the restriction of $\alpha$ to each leaf of $\SF$.

\subsubsection{The space of foliated contact elements}

The following concept will be relevant in the subsequent construction.
\begin{definition}
A \textbf{strong symplectic foliation} is a triple $(M^{m+2n},\SF^{2n}, \omega)$ where $M$ is a smooth manifold, $\SF$ a foliation, and $\omega \in \Omega^2(M)$ a closed 2--form that is symplectic on the leaves of $\SF$.
\end{definition}

Let $(M^{n+m}, \SF^n)$ be a smooth foliation. The cotangent space to the foliation $\pi: T^*\SF \to M$ is an $n$--dimensional bundle over $M$ that carries a natural foliation $\SF^* = \coprod_{\SL \in \SF} \pi^{-1}(\SL)$. Additionally, it is endowed with a canonical $1$--form:
 \[ \lambda_{(p,w)}(v) = w \circ d_{(p,w)}\pi(v), \text{ at a point $(p,w)$, $p \in M$, $w \in T^*_p\SF$.} \]
If $\SL \subset M$ is a leaf of $\SF$ this is nothing but the \textbf{Liouville $1$--form} on $T^*\SL$. Therefore, since $d\lambda$ is a leafwise symplectic form that is globally exact, $(T^*\SF, \SF^*, d\lambda)$ is a strong symplectic foliation.

Fix a leafwise metric $g$ in $M$. Then there is a bundle isomorphism $\#: T^*\SF \to T\SF$. This defines a metric in $T^*\SF$ by setting $g^*(w_1,w_2) = g(\# w_1, \# w_2)$. The presence of $g^*$ allows one to consider the unit cotangent bundle $\mathbb{S}(T^*\SF)$ as a submanifold of $T^*\SF$ transverse to $\SF^*$.

The intersection of $\mathbb{S}(T^*\SF)$ with a leaf $\SL$ is by construction the sphere bundle $\mathbb{S}(T^*\SL)$, which endowed with the form $\lambda$ corresponds to the contact manifold which is called the \textbf{space of oriented contact elements}. Therefore $(\mathbb{S}(T^*\SF), \SF^* \cap \mathbb{S}(T^*\SF), \ker(\lambda))$ is a contact foliation. We call it the \textbf{space of foliated oriented contact elements}.

\begin{lemma}
The Reeb flow in $(\mathbb{S}(T^*\SF), \SF^* \cap \mathbb{S}(T^*\SF), \lambda)$ coincides with the leafwise cogeodesic flow of $g$. 
\end{lemma}

This lemma can be proved just as in the case of contact manifolds (see \cite[Theorem 1.5.2]{Ge}). This construction will be used in Subsection \ref{ssec:tight}.

\subsubsection{The symplectisation of a contact foliation}

\begin{definition}
Let $(M^{2n+1+m}, \SF^{2n+1}, \xi^{2n})$ be a contact foliation. Let $\Theta^{2n+m} \subset TM$ be an extension of $\xi$, and let $\alpha$ be a defining $1$--form for $\Theta$, $\ker(\alpha) = \Theta$. 

We say that 
\[ (\R \times M, \SF_\R = \coprod_{\SL \in \SF} \R \times \SL, \omega = d(e^t\alpha)),  \text{ with $t$ the coordinate in $\mathbb{R}$, }\]
is the \textbf{symplectisation} of $(M, \SF, \xi)$. 
\end{definition}

The symplectisation is another instance of a strong symplectic foliation. Restricted to every individual leaf this is the standard symplectisation of the contact structure on the leaf. 

We are abusing notation and we are writing $\alpha$ for $\pi^* \alpha$, where $\pi: \R \times M \to M$ is the projection onto the second factor. We will also write $\xi$ for the restriction of $(d\pi)^{-1} \xi$ to the level $T(\{t\} \times M)$ and $R$ for the lift of the Reeb vector field $R$ to $\{t\} \times M$. 

Let us also introduce the projection $\pi_\xi: T(\R \times M) \to \xi$ along the $\partial_t$ and $R$ directions.

\section{Several examples}

\subsection{Non--complete Reeb vector fields with no closed orbits} \label{ssec:nonComplete}

It is first reasonable to wonder about the Weinstein conjecture for open manifolds in general. In this direction, not much is known. In \cite{vdBPV} and its sequel \cite{vdBPRV} it is shown that the Weinstein conjecture holds for non--compact energy surfaces in cotangent bundles as long as one imposes certain topology conditions on the hypersurface and certain growth conditions on the hamiltonian, which is assumed to be of mechanical type. 

\begin{proposition} \label{prop:openLeaf}
Let $(N^{2n+1}, \xi)$ be an open contact manifold. Then there is a contact form $\alpha$, $\ker(\alpha) = \xi$, whose (possibly non--complete) associated Reeb flow has no periodic orbits. 
\end{proposition}
\begin{proof}
Fix some small ball $U \subset N$. Modify $\xi$ within $U$ to introduce an overtwisted disc $\Delta$ in the sense of \cite{BEM}. By applying the relative $h$--principle for overtwisted contact structures, there is $\xi_{ot}$ in $N$ that agrees with $\xi$ outside of $U$ and that has $\Delta$ as an overtwisted disc. This new contact structure is homotopic to the original one as almost contact structures.

Let $\{N_i\}_{i \in \mathbb{N}}$ be an exhaustion of $N$ by compact sets, $N_i \subset N_{i+1}$. Fix a non--degenerate contact form $\alpha_{ot}$ for the overtwisted structure $\xi_{ot}$. Its closed Reeb orbits are isolated and countable; moreover, we may assume that no closed orbit is fully contained in $\Delta$. We index them as follows: each compact set $N_i$ is intersected by finitely many closed orbits and hence we write $\{\gamma^i_j\}_{j \in I_i}$ for the collection of closed Reeb orbits intersecting $N_i$ but not $N_{i-1}$. 

Construct a path $\beta: [0,\infty) \to N$, avoiding $\Delta$, that is proper and such that $N \setminus \beta([0,\infty))$ is diffeomorphic to $N$ by a map isotopic to the identity. Then, for each $i$, and each $j \in I_i$, we can construct paths $\beta^i_j: [0,1] \to N_i$ such that the $\beta_j^i$ are all pairwise disjoint, they intersect $\im(\beta)$ only at $\beta^i_j(0) \in \im(\beta)$, they satisfy $\beta^i_j(1) \in \gamma^i_j \cap N_i$, and they avoid $\Delta$. 

Since the images of $\beta$ and the $\beta^i_j$ avoid $\Delta$, we can fix a closed contractible neighbourhood $V$ of $\Delta$ disjoint from them as well. Construct a path $\beta_{ot}: [0,1] \to N$ with $\beta_{ot}(0) \in \partial V$, $\beta_{ot}(1) \in \im(\beta)$ and otherwise avoiding $V$ and all other paths.

Consider the tree $T = \beta \cup \{\cup_{i \in \mathbb{N}, j \in I_i} \beta_j^i \} \cup \beta_{ot}$. Denote by $\nu(T)$ a small closed neighbourhood that deformation retracts onto $T$. We can assume that $N$ is diffeomophic to $N' = N \setminus (\nu(T) \cup V)$ by a diffeomorphism $f: N \to N'$ that is isotopic to the identity.

The embedding $f: (N,\xi) \to (N' \cup V,\xi_{ot})$ has image $N'$ and is covered by a contact bundle homomorphism. This follows because $f$ is isotopic to the identity in $N$ and $\xi$ and $\xi_{ot}$ are homotopic. Now an application of Lemma \ref{coro:BEM} implies that there is an isocontact embedding $\tilde f: (N,\xi) \to (N' \cup V, \xi_{ot})$. The form $\alpha_{ot}$ has no periodic orbits in $N' \cup V$ by construction and hence the pullback form $\alpha = \tilde f^*\alpha_{ot}$ does not either. 
\end{proof}

\begin{remark}
A natural open question is whether it is true that every open contact manifold can be endowed with a contact form inducing a complete Reeb flow with no closed orbits.
\end{remark}

\subsection{The Weinstein conjecture does not hold for contact foliations with all leaves tight} \label{ssec:tight}

We shall construct first a contact foliation with all leaves tight and with periodic orbits lying in the only compact leaf.

\begin{proposition} \label{prop:geodesic}
Let $(\mathbb{S}^3, \SF_{Reeb})$ be the Reeb foliation on the $3$--sphere and let $g$ be the round metric in $\mathbb{S}^3$. Consider the contact foliation $(\mathbb{S}^3 \times \mathbb{S}^1, \lambda_{can})$ on the unit cotangent bundle of $\SF_{Reeb}$. Its only closed Reeb orbits lie in the compact torus leaf. 
\end{proposition}

The proposition is an easy consequence of the following Lemma.

\begin{lemma}
Consider the Riemannian manifold $(\mathbb{R}^2,g)$, where $g$ is of the form $dr \otimes dr + f(r) d\theta \otimes d\theta$, with $f(r)$ an increasing function with $f(r) = r^2$ close to the origin. $(\mathbb{R}^2,g)$ has no closed geodesics.
\end{lemma}
\begin{proof}
Applying the Koszul formula yields the following equations for the Christoffel symbols:
\[ g(\nabla_{\partial_r}\partial_\theta, \partial_\theta) = f'/2 =  \Gamma_{r\theta}^\theta g(\partial_\theta, \partial_\theta) = \Gamma_{r\theta}^\theta f, \]
\[ g(\nabla_{\partial_\theta}\partial_r, \partial_\theta) = f'/2 =  \Gamma_{\theta r}^\theta g(\partial_\theta, \partial_\theta) = \Gamma_{\theta r}^\theta f, \]
\[ g(\nabla_{\partial_\theta}\partial_\theta, \partial_r) = -f'/2 =  \Gamma_{\theta \theta}^r g(\partial_r, \partial_r) = \Gamma_{\theta \theta}^r. \]
And hence the geodesic equations read:
\[ \overset{..}{r} = f'\overset{.}{\theta}^2, \]
\[ \overset{..}{\theta} = -\log(f)'\overset{.}{\theta}\overset{.}{r}. \]

If at any point $\overset{.}{\theta}=0$, then $\overset{.}{\theta} = 0$ for all times and $\overset{.}{r}$ is a constant. This situation corresponds to radial lines. 

All other geodesics have always $\overset{.}{\theta} \neq 0$ and hence $\overset{..}{r} > 0$. In particular, as soon as a geodesic has $\overset{.}{r} \geq 0$ at some point, it will have $\overset{.}{r} > 0$ for all the points in the forward orbit and hence it will not close up. 

For a geodesic to close up we deduce then that it must have $\overset{.}{r} < 0$ for all times, but then it cannot close up either.
\end{proof}

\begin{proof}[Proof of Proposition \ref{prop:geodesic}]
Consider $\mathbb{S}^3$ lying in $\mathbb{C}^2$, with coordinates $(z_1,z_2) = (r_1,\theta_1,r_2,\theta_2)$. The Reeb foliation can be assumed to have the Clifford torus $|z_1|^2 = |z_2|^2 = 1/2$ as its torus leaf. One of the solid tori, denote it by $T$, corresponds to $\{|z_1|^2 \leq 1/2, |z_2|^2 = 1 - |z_1|^2\}$ and the other one is given by the symmetric equation. Let us multiple cover the torus $T$ with the map $\phi: \mathbb{R} \times \mathbb{D}^2_{1/\sqrt{2}} \rightarrow T$ given by $\phi(s,r,\theta) = (r,\theta,\sqrt{1-r^2},s)$. For all purposes we can work in $\mathbb{R} \times \mathbb{D}^2_{1/\sqrt{2}}$, which is the universal cover of the torus, and hence we shall do so. 

The restriction of the flat metric of $\mathbb{C}^2$
\[ g = \sum_{i=1,2} dr_i \otimes dr_i + r_i^2 d\theta_i \otimes d\theta_i \]
to $\mathbb{S}^3$ is precisely the round metric. In the parametrisation of $T$ given above it reads as:
\[ \phi^* g = \dfrac{1}{1-r^2} dr \otimes dr + r^2 d\theta \otimes d\theta + (1-r^2) ds \otimes ds. \]
Which in particular readily shows that the metric induced in the Clifford torus is flat.

Consider the embeddings 
\[ \psi_c: \mathbb{R}^2 \rightarrow \mathbb{R} \times \mathbb{D}^2_{1/\sqrt{2}} \]
\[ \psi_c(\rho,\theta) = (f(\rho)+c,\dfrac{\rho}{\sqrt{2}(1+\rho)},\theta), \]
with $f: \mathbb{R} \to \mathbb{R}$ a smooth increasing function that agrees with $\rho^2$ near the origin and with the identity away from it. They realise the non--compact leaves of the Reeb foliation in $T$. It is clear that the leafwise metric is of the form
\[ \psi_c^* \phi^* g = h_1(\rho) d\rho \otimes d\rho + h_2(\rho) d\theta \otimes d\theta \]
with $h_2(\rho)$ increasing and converging to $1/2$ as $\rho \rightarrow \infty$ and $h_1(\rho)$ bounded from above and behaving as $O(\rho)$ near the origin. 

At every point of $\mathbb{R}^2$ a vector field $X$ pointing radially and of unit length can be defined. The properties of $h_1$ imply that $X$ is complete and following $X$ yields a reparametrisation $\Phi: \mathbb{R}^2 \rightarrow \mathbb{R}^2$ that satisfies:
\[ \Phi^*\psi_c^* \phi^* g = d\rho \otimes d\rho + \tilde{h}(\rho) d\theta \otimes d\theta \]
with $\tilde{h}(\rho)$ still increasing and converging to $1/2$ as $\rho \rightarrow \infty$. Now the Lemma yields the result. 
\end{proof}

\begin{remark}
Taking the universal cover of a leaf yields the standard tight $\mathbb{R}^3$, so all leaves are tight.
\end{remark}

One can actually construct a contact foliation with no periodic orbits of the Reeb.

\begin{proposition} \label{prop:noGeodesics}
Consider the manifold $\mathbb{T}^3$, endowed with the Euclidean metric $g$, and the foliation $\SF$ by planes given by two rationally independent slopes. The space of foliated cooriented contact elements $\mathbb{S}(T^*\SF)$ has no closed Reeb orbits.
\end{proposition}
\begin{proof}
Let $\SL$ be any leaf of $\SF$. $\SL$ is diffeomorphic to $\mathbb{R}^2 \times \mathbb{S}^1$ and its universal cover of is the standard tight $\mathbb{R}^3$. Hence it is a tight contact manifold. Since the restriction of $g$ to $\SL$ is Euclidean, there are no closed geodesics on $\SL$ and hence no closed Reeb orbits in its sphere cotangent bundle.
\end{proof}

\subsection{A sharp example. Overtwisted leaves with no closed orbits} \label{ssec:sharp}

\subsubsection{$\mathbb{R}^3$ overtwisted at infinity with no closed orbits}

Consider the following $1$--form in $\mathbb{R}^3$ in cylindrical coordinates:
\[ \alpha = \cos(r) dz + (r\sin(r) + f(z)\phi(r))d\theta \]
If $f(z)\phi(r) = 0$ identically, this is the standard form $\alpha_{ot}$ for the contact structure $\xi_{ot}$ that is overtwisted at infinity. We well henceforth assume that $f(z)\phi(r)$ is $C^1$--small, and therefore $\alpha$ will be a contact form as well. In particular, by Proposition \ref{prop:El93}, the contact structure it defines is contactomorphic to $\xi_{ot}$. Let us compute:
\[ d\alpha = -\sin(r) dr \wedge dz + [\sin(r) + r\cos(r) + \phi'(r)f(z)]dr \wedge d\theta + f'(z)\phi(r) dz \wedge d\theta \]
whose kernel, away from the origin, is spanned by:
\[ X = -f'(z)\phi(r) \partial_r + [\sin(r) + r\cos(r) + \phi'(r)f(z)] \partial_z + \sin(r) \partial_\theta. \]
It is easy to check that $\alpha(X) > 0$ far from the origin, and hence the Reeb is a positive multiple of $X$. 

Assume that $\phi(r)$ is a monotone function that is identically $0$ close to $0$ and identically $1$ in $[\delta, \infty)$, for $\delta > 0$ small. Then the Reeb vector field in the origin is $\partial_z$, and remains almost vertical nearby. Assume further that $f$ is strictly decreasing and $C^1$ small. Then the Reeb has a positive radial component away from the origin. We conclude that it has no closed orbits. 

\subsubsection{$\mathbb{S}^2 \times \mathbb{R}$ overtwisted at infinity with no closed orbits}

Consider coordinates $(z,\theta; s)$ in $\mathbb{S}^2 \times \mathbb{R}$, with $z \in [0, 2\pi]$, and construct the following $1$--form:
\[ \lambda_0 = \cos(z)ds + z(z-2\pi)\sin(z)d\theta. \]
It is easy to see that it is a contact form that defines two families of overtwisted discs sharing a common boundary: $\{ z \in [0,\pi], s = s_0\}$ and $\{ z \in [\pi,2\pi], s = s_0\}$. It is therefore overtwisted at infinity. 

$\lambda_0$ defines two cylinders comprised of Reeb orbits: $\{z = \pi/2\}$ and $\{z = 3\pi/2\}$. Therefore, proceeding like in the previous example, we will add a small perturbation that gets rid of these orbits. Consider the form:
\[ \lambda =  \cos(z)ds + [z(z-2\pi)\sin(z) + f(s)\phi(z)]d\theta. \]
Here we require for $\phi(z)$ to be constant close to the points $0$, $\pi/2$, $\pi$, $3\pi/2$ and $2\pi$, to satisfy:
\[ \phi(0) = \phi(\pi) = \phi(2\pi) = 0, \quad \phi(\pi/2) = \phi(3\pi/2) = 1 \]
and to be monotone in the subintervals inbetween. We assume that $f$ is strictly monotone and $C^1$ small. Computing:
\[ d\lambda =  -\sin(z)dr \wedge ds + [ (z-2\pi)\sin(z) + z\sin(z) + z(z-2\pi)\cos(z) + f(s)\phi'(z)]dz \wedge d\theta + f'(s)\phi(z) ds \wedge d\theta \]
so the Reeb flow is a multiple of:
\[ X = -f'(s)\phi(z) \partial_z + [ (z-2\pi)\sin(z) + z\sin(z) + z(z-2\pi)\cos(z) + f(s)\phi'(z)] \partial_s + \sin(z) \partial_\theta. \]
Near $z = 0, \pi, 2\pi$ the Reeb is very close to $\pm \partial_s$. Away from those points, it has a non--zero $z$--component. It follows that it cannot have closed orbits.

\subsubsection{Constructing the foliation}

Consider $\mathbb{S}^2 \times \mathbb{S}^1 \times \mathbb{S}^1$ with coordinates $(z,\theta;s,t)$, $t \in [0,2]$. It can be endowed with the following $1$--form:
\[ \tilde\lambda =  \cos(z)ds + [z(z-2\pi)\sin(z) + F(t)\phi(z)]d\theta, \]
with $F$ strictly increasing in $(0,1)$, strictly decreasing in $(1,2)$, $C^1$--small and having vanishing derivatives to all orders in $\{0,1\}$. $\phi$ is the bump function defined in the previous subsection.

Let $\Phi: \mathbb{S}^1 \to \mathbb{S}^1$ be a diffeomorphism of the circle that fixes $\{0,1\}$ and no other points, is strictly increasing in $(0,1)$ as a map $(0,1) \to (0,1)$, and is strictly decreasing in $(1,2)$ as a map $(1,2) \to (1,2)$. $\Phi$ defines a foliation $\SF_\Phi$ on $\mathbb{S}^2 \times \mathbb{S}^1 \times \mathbb{S}^1$ called the \textbf{suspension} of $\Phi$.

$\SF_\Phi$ can be constructed as follows. Find a family of functions $\Phi_s: \mathbb{S}^1 \to \mathbb{S}^1$, $s \in [0,1]$, satisfying:
\begin{equation} \begin{cases}
 \Phi_0 = \Id, \quad \Phi_1 = \Phi,  \\
 \text{ the map } s \to \Phi_s(t) \text{ is strictly increasing in $(0,1)$ and strictly decreasing in $(1,2)$}, \\
 \left.\dfrac{\partial}{\partial s}\right|_{s=1} \Phi_s(t) = \left.\dfrac{\partial}{\partial s}\right|_{s=0} \Phi_s(\Phi_1(t)) \quad \text{ for all } t. 
\end{cases} \end{equation}
Then the curves $\gamma_t(s) = (s,\Phi_s(t))$, induce a foliation in $[0,1] \times \mathbb{S}^1$ which glues to yield a foliation by curves in the $2$--torus. $\SF_\Phi$ is the lift of such a foliation.

The leaves of the foliation in the $2$--torus are obtained by concatenating the segments $\gamma_t$. $\gamma_0$ and $\gamma_1$ yield closed curves $\tilde\gamma_0$ and $\tilde\gamma_1$. All other curves are diffeomorphic to $\mathbb{R}$, and we denote them by $\tilde\gamma_t(s) = (s,h_t(s))$, $t \in (0,1) \cup (1,2)$. By our assumption on $\Phi_s$, the functions $h_t$ are strictly increasing if $t \in (0,1)$ and strictly decreasing if $t \in (1,2)$. Observe that the non--compact leaves accumulate against the two compact ones. 

The contact structure in the compact leaves $\mathbb{S}^2 \times \tilde\gamma_t$, $t=0,1$, is given by 
\[ \cos(z)ds + z(z-2\pi)\sin(z)d\theta. \] 
In particular, they both have infinitely many closed orbits.

The contact structure in the non compact leaves $\mathbb{S}^2 \times \tilde\gamma_t$, $t \in (0,1) \cup (1,2)$, reads 
\[ \cos(z)ds + [z(z-2\pi)\sin(z) + F(h_t(s))\phi(z)]d\theta \] 
Since $F \circ h_t$ is non--zero, strictly monotone and $C^1$--small, it is of the form described in the previous section. It follows that they have no periodic orbits. 

\begin{remark}
In this example all leaves involved are overtwisted. Further, the non--compact leaves are overtwisted at infinity. It would be interesting to construct an example of a contact foliation where the non--compact leaves are overtwisted, the leaves in their closure are tight and the only periodic orbits appear in the tight leaves. 
\end{remark}

\section{$J$--holomorphic curves in the symplectisation of a contact foliation}

In this section we generalise the standard setup for moduli spaces of pseudoholomorphic curves to the foliated setting. The main result is Theorem \ref{thm:removalSingularities}, which deals with the removal of singularities. The proof is standard and closely follows that of \cite{Ho}, and indeed the only essential difference lies in the fact that, although the leaves might be open, they live inside a compact ambient manifold, so the Arzel\'a--Ascoli theorem can still be applied when carrying out the bubbling analysis.

\subsection{Setup} 

Consider the contact foliation $(M^{m+2n+1}, \SF^{2n+1}, \xi^{2n})$, with extension $\Theta^{2n+m}$ given by a $1$--form $\alpha$, and write $(\R \times M, \SF_\mathbb{R}, \omega)$ for its symplectisation. 

\subsubsection{The space of almost complex structures}

The symplectic bundle $(\xi, d\alpha)$ can be endowed with a complex structure compatible with $d\alpha$, which we denote by $J_\xi$. The space of such choices is non--empty and contractible. $J_\xi$ induces a unique $\mathbb{R}$--invariant leafwise complex structure, $J \in \End(T\SF_\mathbb{R})$, $J^2 = -\Id$, as follows:
\[ J|_{\xi} = J_\xi \]
\[ J(\partial_t) = R \]
Observe that $J$ is \textbf{compatible} with $\omega$, and hence they define a metric, which turns each leaf of the symplectisation into a manifold which is not complete. Instead, we shall consider the better behaved $\mathbb{R}$--invariant leafwise riemannian metric $g$ in $\R \times \SF$ given by:
\[ g = dt \otimes dt + \alpha \otimes \alpha + d\alpha(J_\xi \circ \pi_\xi , \pi_\xi ). \]

\subsubsection{$J$--holomorphic curves}

Let $(S, i)$ be a Riemann surface, possibly with boundary. A map satisfying
\begin{equation} \begin{cases} \label{eq:holCurves}
  F: (S,i) \to (\R \times M,J) \\
	dF(TS) \subset \SF_\mathbb{R} \\
	J \circ dF = dF \circ i
\end{cases} \end{equation}
is called a parametrised \textbf{foliated $J$--holomorphic curve}. The second condition implies that $F(S)$ is contained in a leaf $\R \times \SL$ of $\SF_\mathbb{R}$. Indeed, $J$ is an almost complex structure in the open manifold $\R \times \SL$, and $F$, regarded as a map into $\R \times \SL$, is a \textbf{$J$--holomorphic curve} in the standard sense.

By our choice of $J$, there is an $\mathbb{R}$--action on the space of foliated $J$--holomorphic curves given by translation on the $\mathbb{R}$ term of $\R \times M$.

\subsubsection{Foliated $J$--holomorphic planes and cylinders}

A solution of Equation (\ref{eq:holCurves})
\[ F = (a,u): (\mathbb{C},i) \rightarrow (\R \times M,J) \]
is called a \textbf{foliated $J$--holomorphic plane}. If we write $\mathcal{M}_J^\SF$ for the space of such maps, it is clear that the space of complex automorphisms of $\mathbb{C}$ acts on it by its action on the domain.

$\mathcal{M}_J^\SF$ is non--empty. Every Reeb orbit $\gamma: \mathbb{R} \rightarrow M$ has an associated foliated $J$--holomorphic plane given by
\[ F(s,t) = (s,\gamma(t))\quad \text{ where $z = s+it$ are the standard complex coordinates in $\mathbb{C}$.}\]
We call these the \textsl{trivial} solutions. 

Similarly, a solution of Equation \ref{eq:holCurves} 
\[ F = (a,u): (-\infty, \infty) \times \mathbb{S}^1 \rightarrow \R \times M \]
is called a \textbf{foliated $J$--holomorphic cylinder}. We let $(s,t)$ be the coordinates in the cylinder and its complex structure to be given by $i(\partial_s) = \partial_t$. A closed Reeb orbit $\gamma: \mathbb{S}^1 \to M$, gives a \textsl{trivial} cylinder $F(s,t) = (s,\gamma(t))$.

Recall that the cylinder $(-\infty, \infty) \times \mathbb{R}$ is biholomorphic to $\mathbb{C} \setminus \{0\}$ by the exponential, and for convenience we will often consider both domains interchangeably. In particular, given some foliated $J$--holomorphic plane, we could define a foliated $J$--holomorphic cylinder by introducing a pucture in the domain. Therefore, we say that a foliated $J$--holomorphic map
\[ F=(a,u): \mathbb{C} \setminus \{0\} \to \R \times M \]
can be \textbf{extended} over zero (or $\infty$) if there is a foliated $J$--holomorphic map with domain $\mathbb{C}$ (resp. the puctured Riemann sphere $\hat{\mathbb{C}} \setminus \{0\})$ that agrees with $F$ in $\mathbb{C} \setminus \{0\}$.

\subsubsection{Energy}

After introducing the \textsl{trivial} foliated $J$--holomorphic curves, we would like to introduce an \textsl{energy constraint} that singles out more interesting solutions of Equation \ref{eq:holCurves}. This leads us to the following definitions. 

\begin{definition}
Consider the space of functions 
\[ \Gamma = \{ \phi \in C^\infty(\mathbb{R}, [0,1])| \quad \phi' \geq 0 \} \]
Let $F: S \to \R \times M$ be a foliated $J$--holomorphic curve.

Its \textbf{energy} is defined by:
\begin{align} \label{eq:energy}
E(F) = \sup_{\phi \in \Gamma} \int_S F^* d(\phi\alpha).
\end{align}

Its \textbf{horizontal energy} is defined by:
\begin{align} \label{eq:horizontalEnergy}
E^h(F) = \int_S F^* d\alpha.
\end{align}
\end{definition}

Trivial solutions correspond to the following general phenomenon.
\begin{lemma} \label{lem:zeroHorizontalEnergy}
Let $F=(a,u): (S,i) \to (\R \times M,J)$ be a foliated $J$--holomorphic curve. $E^h(F) = 0$ if and only if $\im(F) \subset \R \times \gamma$, where $\gamma$ is a Reeb orbit. 
\end{lemma}
\begin{proof}
Given a ball $U \subset S$ find complex coordinates $(s,t)$. Then:
\[ \int_U F^* d\alpha = \int_U d\alpha(u_s, u_t) ds \wedge dt = \int_U d\alpha(u_s, J u_s) ds \wedge dt = \]
\[ \int_U d\alpha(\pi_\xi u_s, \pi_\xi \circ J u_t) ds \wedge dt = \int_U |\pi_\xi u_s| ds \wedge dt \]
and since
\[ E^h(F) = \int_S F^* d\alpha = \int_S u^* d\alpha \]
the claim follows.
\end{proof}

The following lemma states that cylinders with finite energy that cannot be extended to planes have to be necessarily trivial and hence imply the existence of a Reeb orbit.

\begin{lemma} \label{lem:trivialCylinder}
Let $F$ be a foliated $J$--holomorphic map
\[ F=(a,u): \mathbb{C} \setminus \{0\} \to \R \times M \]
satisfying $E(F) < \infty$ and $E^h(F) = 0$. If $F$ cannot be extended over its punctures, then $t \to u(e^{2\pi it})$, $t \in [0,1]$, is a parametrised closed Reeb orbit.
\end{lemma}
\begin{proof}
By Lemma \ref{lem:zeroHorizontalEnergy}, we know that there is some Reeb orbit $\gamma$ (not necessarily closed) such that $\im(F) \subset \R \times \gamma$. We can identify the universal cover of $\R \times \gamma$ with $\mathbb{C}$ with its standard complex structure.

We claim that $\gamma$ is a closed orbit and that $F$ is a non contractible map into $\R \times \gamma$. Assuming otherwise, regard $F$ as a holomorphic map $f: \mathbb{C} \setminus \{0\} \to \mathbb{C} \subset \hat{\mathbb{C}}$. As such, its punctures are either removable or essential singularities. They cannot be removable singularities with values in $\mathbb{C}$ by assumption. 

If $f$ has a removable singularity that is a pole, a neighbourhood of the pucture branch covers a neighbourhood of $\infty$ in the Riemann sphere. In particular, there is a band $[a,b] \times \R \subset \im(f) \subset \mathbb{C}$, with $a < b$ large enough. This contradicts the assumption that $E(F)$ was finite. 

If $f$ has an essential singularity, then Picard's great theorem states that every point in $\mathbb{C}$, except possibly one, is contained in $\im(f)$. Again, this contradicts the assumption that $E(F)$ was finite. 

We deduce that $\gamma$ is a closed orbit and that $F$ is a non--contractible map into the cylinder $\R \times \gamma$. The exponential is a biholomorphism between the cylinder and $\mathbb{C} \setminus \{0\}$, so now we regard $F$ as a holomorphic map $h: \mathbb{C} \setminus \{0\} \to \mathbb{C} \setminus \{0\}$. 

Suppose one of the punctures was an essential singularity for $h$. Since $h$ has no zeroes or poles, Picard's theorem states that all other points in the Riemann sphere have infinitely many preimages by $h$. This contradicts $E(F) < \infty$. 

Therefore, $h$ can be extended over its punctures to be zero or $\infty$. $h$ is then a meromorphic function over the Riemann sphere, and hence it is nothing but the quotient of two polynomials. By our assumption that there are no other zeroes or poles this implies that $h(z) = az^k$, for some $k \in \mathbb{Z} \setminus \{0\}$, $a \in \mathbb{C}$. This shows that $t \to u(e^{2\pi it})$ parametrises the $k$--fold cover of $\gamma$.
\end{proof}

Exactly the same analysis yields the following lemma.

\begin{lemma} \label{lem:trivialPlane}
Let $F$ be a foliated $J$--holomorphic map
\[ F=(a,u): \mathbb{C} \to \R \times M \]
satisfying $E^h(F) = 0$. Then either $F$ is the constant map or $E(F) = \infty$.
\end{lemma}
\begin{proof}
Let $\gamma$ be the Reeb orbit such that $\im(F) \subset \R \times \gamma$. By taking the universal cover of $\R \times \gamma$, regard $F$ as a map $\mathbb{C} \to \mathbb{C}$, as in Lemma \ref{lem:trivialCylinder}. Now study the extension problem of $F$ to $\infty$. If it corresponds to a removable singularity with values in $\mathbb{C}$, then $F$ is the constant map. Otherwise, if it is either a pole or a non--removable singularity, it has infinite energy.
\end{proof}

\subsubsection{Riemannian and symplectic area}

In the case of compact symplectic manifolds, there is an interplay between the symplectic area of a $J$--holomorphic curve and its riemannian area for the metric given by the symplectic form and the compatible almost complex structure. 

In our case, $g$ is not of that form. Rather, it is $\mathbb{R}$--invariant, while $\omega$ is not: $\mathbb{R}$--translations of the same $J$--holomorphic curve have different symplectic energy and indeed there are no universal constants relating the $\omega$--area and the $g$--area. 

However, $E$ and $E^h$ are invariant under the $\mathbb{R}$-action. Given $F$, a foliated $J$--holomorphic curve, let $\area_g(F)$ be its riemannian area in terms of $g$, and let $\area_{\omega_\phi}(F)$ be its symplectic area in terms of $\omega_\phi = d(\phi\alpha)$.

\begin{lemma} \label{lem:noSphereBubbling}
Let $F=(a,u): (S,i) \to (\R \times M,J)$ be a parametrised foliated $J$--holomorphic curve. Then, if $a$ is bounded below and above:
\[ \area_g(F) < C\area_\omega(F) < C'\int_{\partial S} \alpha, \]
for some constants $C, C'$ depending only on the upper and lower bounds of $a$. 
\end{lemma}
\begin{proof}
Consider $a_0$ and $a_1$ satisfying $a_0 < a < a_1$. Let $\phi(t) = \frac{t-a_0}{3(a_1-a_0)} + 1/3$ in $[a_0,a_1]$ and belonging to $\Gamma$. Then $\omega_\phi$ is a symplectic form in $[a_0,a_1] \times M$ and $J$ is $\omega_\phi$--compatible. Since $0 < D < \phi, \phi' < D' < \infty$, there are universal constants relating the metrics $g$ and $g_\phi = \omega_\phi(-,J-)$ in $[a_0,a_1] \times M$.

Since $J$ is $\omega_\phi$--compatible, $F$ being $J$--holomorphic implies that $\area_{g_\phi}(F) = \area_{\omega_\phi}(F)$, and the first inequality follows. The second inequality follows by applying Stokes. 
\end{proof}

An immediate consequence of Lemma \ref{lem:noSphereBubbling} is that there cannot be \textsl{closed} foliated $J$--holomorphic curves in $\R \times M$.

\subsection{Bubbling}

As we shall see in Section \ref{sec:main}, the way in which we will prove the existence of a periodic orbit of the Reeb vector field will be by constructing a $1$--dimensional moduli of pseudoholomorphic discs that necessarily will be open in one of its ends. The following lemma shows that the reason for it to be open must be that the gradient is not uniformly bounded for all discs in the moduli. 

\begin{proposition} \label{prop:ArzelaAscoli}
Fix $\SL$ a leaf of $\SF$. Let $W \subset \R \times \SL$ be a totally real compact submanifold, possibly with boundary.

Let $(S,i)$ be a compact Riemann surface with boundary. Consider the sequence of foliated $J$--holomorphic maps
\[ F_k: (S, \partial S) \to (\R \times \SL, W), \quad k \in \mathbb{N}. \]

Suppose that there is a uniform bound $||dF_k|| <  C < \infty$. Then there is a subsequence $F_{k_i}$, $k_i \to \infty$, convergent in the $C^\infty$--topology to a foliated $J$--holomorphic map
\[ F_\infty: (S, \partial S) \to (\R \times M, W) \]
\end{proposition}
\begin{proof}
Observe that since we have a uniform gradient bound and $F_k(\partial S) \subset W$, for all $k$, it necessarily follows that the images of all the $F_k$ lie in a compact subset of $\R \times \SL$. Then one can proceed as in the standard case to prove $C^\infty$ bounds from $C^1$ bounds and then apply Arzel\'a-Ascoli to conclude.
\end{proof}

\begin{remark}
The same statement holds for surfaces without boundary as long as one imposes for the images of all the $F_k$ to lie in a compact set of the leaf. 
\end{remark}

Proposition \ref{prop:ArzelaAscoli} suggests that we should study sequences of maps
\[ F_k: (S, \partial S) \to (\R \times \SL, W), \quad k \in \mathbb{N} \]
in which $||dF_k||$ is not uniformly bounded. We have to consider two separate cases.

\subsubsection{Plane bubbling}

\begin{proposition} \label{prop:planeBubbling}
Consider a sequence of foliated $J$--holomorphic curves
\[ F_k: (S, \partial S) \to (\R \times \SL, W), \quad k \in \mathbb{N} \]
and a corresponding sequence of points $q_k$ in $S$ having $M_k = ||d_{q_k}F_k|| \to \infty$ and converging to a point $q \in S$. 

Suppose that there is an uniform bound $E(F_k) < C < \infty$. If $\dist(q_k,\partial S)M_k \to \infty$, there is a foliated $J$--holomorphic plane
\[ F_\infty: \mathbb{C} \to \R \times \SL' \]
with $E(F_\infty) < C$, where $\SL'$ is a leaf in the closure of $\SL$. 
\end{proposition}
\begin{proof}
After possibly modifying the $q_k$ slightly, there are charts
\[ \phi_k: \mathbb{D}^2(R_k) \to S \]
\[ \phi_k(z) = q_k + \dfrac{z}{M_k} \]
with $R_k < \dist(q_k,\partial S)M_k$, $R_k \to \infty$, $R_k/M_k \to 0$ and $||d(F_k \circ \phi_k)|| < 2$ -- this last condition is achieved by the so called Hofer's lemma, see \cite[Lemma 26]{Ho}. 

The maps $F_k \circ \phi_k$ have $C^1$ bounds by construction, but they have no $C^0$ bounds. By our construction of $J$, the vertical translation of a $J$--holomorphic map is still $J$--holomorphic and hence we can compose with a vertical translation $\tau_k$ guaranteeing that $\tau_k \circ F_k \circ \phi_k$ takes the point $0$ to the level $\{0\} \times \SL$. Then, for every compact subset $\Omega \subset \mathbb{C}$, the maps $\tau_k \circ F_k \circ \phi_k: \Omega \to \R \times M$ are equicontinuous and bounded -- note that this is where we use that $\SL$ lies inside the compact manifold $M$. 

Recall that having uniform $C^1$ bounds implies that we have uniform $C^\infty$ bounds. Hence, an application of Arzel\'a--Ascoli shows that a subsequence converges in $C^\infty_{loc}$ to a map $F_\infty: \mathbb{C} \to \R \times M$ that must be foliated and $J$--holomorphic, but not necessarily lying in $\R \times \SL$, but maybe in some new leaf $\R \times \SL'$. 

Note that the energy of the map $\tau_k \circ F_k \circ \phi_k$ is bounded above by that of $F_k$. Since we have uniform bounds for the energy of the $F_k$, we have uniform energy bounds for the maps $\tau_k \circ F_k \circ \phi_k$ and hence for their limit $F_\infty$. Note that $F_\infty$ is necessarily non constant, since $||d_0F_\infty|| = 1$ by construction. In particular, it has non--zero energy. 
\end{proof}

\begin{remark}
We say that the map $F_\infty$ as given in the proof is called a \textbf{plane bubble}. If the map $F_\infty$ could be extended over the pucture to a map with domain the Riemann sphere $\mathbb{S}^2$, this would yield a contradiction with Lemma \ref{lem:noSphereBubbling}.
\end{remark}

\subsubsection{Disc bubbling}

\begin{proposition} \label{prop:discBubbling}
Consider a sequence of foliated $J$--holomorphic curves
\[ F_k: (S, \partial S) \to (\R \times \SL, W), \quad k \in \mathbb{N} \]
and a corresponding sequence of points $q_k$ in $S$ having $M_k = ||d_{q_k}F_k|| \to \infty$ converging to a point $q \in S$. 

Suppose that there is an uniform bound $E(F_k) < C < \infty$. If $\dist(q_k,\partial S)M_k$ is uniformly bounded from above, there is a foliated $J$--holomorphic disc
\[ F_\infty: (\mathbb{D}^2, \partial \mathbb{D}^2) \to (\R \times \SL,W) \]
with $E(F_\infty) < C$.
\end{proposition}
\begin{proof}
Since we are assuming that $W$ is compact, the usual rescaling argument for the disc bubbling goes through and yields a punctured disc bubble lying in $\R \times \SL$ and having bounded gradient. Then, the standard removal of singularities gives a disc bubble $F_\infty$.
\end{proof}

\subsection{Removal of singularities}

The aim of this subsection is to prove the following result, which is one of the key ingredients for proving Theorem \ref{thm:main}.

\begin{theorem}[Removal of singularities] \label{thm:removalSingularities}
Let $F=(a,u): \mathbb{D}^2 \setminus \{0\} \to \R \times \SL \subset \R \times M$ be a $J$--holomorphic curve with $0 < E(F) < \infty$, $\SL$ a leaf of $\SF$. 

Then, either $F$ extends to a $J$--holomorphic map over $\mathbb{D}^2$ or for every sequence of radii $r_k \to 0$ the curves $\gamma_{r_k}(s) = u(e^{r_k+is})$ converge in $C^\infty$ --possibly after taking a subsequence-- to a parametrised closed Reeb orbit lying in the closure of $\SL$. 
\end{theorem}
\begin{proof}[Proof of Theorem \ref{thm:removalSingularities}]
Let us state the problem in terms of cylinders. Identify $\mathbb{D}^2 \setminus \{0\}$ with $[0, \infty) \times \mathbb{S}^1$ by using the biholomorphism $-\log(z)$, and regard $F$ as a foliated $J$--holomorphic map $[0, \infty) \times \mathbb{S}^1 \to \R \times M$. Then, the following maps are foliated $J$--holomorphic:
\[ F_k=(a_k,u_k): [-R_k/2, \infty) \times \mathbb{S}^1 \to \R \times M \]
\[ F_k(s,t) = (a(s+R_k,t) - a(R_k,0),u(s+R_k,t)) \]
and by assumption they have a uniform bound $E(F_k) < C < \infty$ and $\lim_{k \to \infty} E^h(F_k) = 0$. Here $R_k = -\log(r_k) \to \infty$.

Suppose that the gradient was not uniformly bounded for the family $F_k$. We can then find a sequence of points $q_k \in [0, \infty) \times \mathbb{S}^1$ escaping to infinity and satisfying $|d_{q_k}F| \to \infty$. Then we are under the assumptions of Proposition \ref{prop:planeBubbling}, and this yields a plane bubble $G: \mathbb{C} \to \R \times M$ with $E^h(G) = 0$, which must lie on top of a Reeb orbit by Lemma \ref{lem:zeroHorizontalEnergy}. By our bubbling analysis, it cannot be constant, since its gradient at the origin is $1$, which is a contradiction with it having $E(G) < \infty$, by Lemma \ref{lem:trivialPlane}.

We conclude that the family $F_k$ has uniform $C^1$ bounds and hence uniform $C^\infty$ bounds. By construction $a_k(0,0) \in \{0\} \times M$, which means that we have uniform $C^0$ bounds on every compact subset of $(-\infty, \infty) \times \mathbb{S}^1$ --here is where we use the compactness of $M$. Arzel\'a-Ascoli implies that --after possibly taking a subsequence-- the maps $F_k$ converge in $C^\infty_{loc}$ to a map $F_\infty: (-\infty, \infty) \times \mathbb{S}^1 \to \R \times M$ with $E(F_\infty) < \infty$ and $E^h(F_\infty) = 0$. 

Observe that 
\[ \lim_{r \to 0} \int_{\gamma_r} \alpha = \int_{\gamma_1} \alpha - \int_{\mathbb{D}^2 \setminus \{0\}} d\alpha. \]
If this limit is zero, then the argument above shows that the $\gamma_r$, $r \to 0$, tend to the constant map in the $C^\infty$ sense, and hence $F$ extends to a map over $\mathbb{D}^2$. Assuming otherwise, it is clear that $F_\infty$ cannot be the constant map and hence Lemma \ref{lem:trivialCylinder} implies the conclusion. 
\end{proof}

\section{Existence of contractible periodic orbits in the closure of an overtwisted leaf} \label{sec:main}

After setting up the study of foliated $J$--holomorphic curves in the previous section and dealing with its compactness issues, we use this machinery to conclude the proof of Theorem \ref{thm:main}. The setting of the theorem is as follows: $(M^{m+3}, \SF^3, \xi^2)$ is a contact foliation with $\Theta^{2+m}$ an extension given by a $1$--form $\alpha$. We write $(\R \times M, \SF_\mathbb{R}, \omega)$ for its symplectisation. $\SL^3$ is a leaf of $\SF$.

\subsection{The Bishop family}

The following results have a local nature and hence do not depend on whether $\SL$ is compact or not. Their proofs can be found in \cite{Ho}.

\subsubsection{The Bishop family at an elliptic point}

If $(\SL, \xi)$ is an overtwisted manifold, let $\Sigma$ be an overtwisted disc for $\xi$. Otherwise, if $\pi_2(\SL) \neq 0$, let $\Sigma$ be some sphere realising a non--zero class in $\pi_2$. Assume, after a small perturbation, that the characteristic foliations are as described in Subsection \ref{sssec:convex} in Exercises \ref{ex:ot} and \ref{ex:tight} and Theorem \ref{thm:EGF}. Denote by $\Gamma_\Sigma$ the set of singular points of the characteristic foliation of $\Sigma$.

Let $p \in \Gamma_\Sigma$, a nicely elliptic point. The maps satisfying:
\begin{equation} \begin{cases} \label{eq:holDiscs}
	F = (u, a): (\mathbb{D}^2, \partial \mathbb{D}^2) \to (\R \times \SL, \{0\} \times \Sigma) \\
	dF \circ i = J \circ dF, \\
	\wind(F, p) = \pm 1, \\
	\ind(F) = 4,
\end{cases} \end{equation}
will be called the \textbf{Bishop family}. $\wind(F, p)$ refers to the winding number of $F(\partial \mathbb{D}^2$) around the elliptic point $p$. 

The condition $\ind(F) = 4$ is implied by the other assumptions. It means that the linearised Cauchy--Riemann operator at $F$ has index $4$, and hence, if there is transversality, the solutions of Equation \ref{eq:holDiscs} close to $F$ form a smooth $4$--dimensional manifold. Since the Mobius transformations of the disc have real dimension $3$, this implies that the image of $F$ is part of a $1$--dimensional family of distinct discs. 

The Bishop family is not empty under some integrability assumptions.

\begin{proposition} \label{prop:Bishop} \emph{(\cite{Bi}, \cite[Section 4.2]{Ho})}
For a suitable choice of $J_\xi$, $J$ is integrable close to $p$. Then there is a smooth family of maps $F_s$, $s \in [0,\varepsilon)$, with $F_0(z) = p$ and $F_s$, $s>0$, disjoint embeddings satisfying Equation \ref{eq:holDiscs}.

Additionally, there is a small neighbourhood $U$ of $p$ such that any other disc satisfying Equation \ref{eq:holDiscs} and interesecting $U$ is a reparametrisation of one of the $F_s$.
\end{proposition}

\subsubsection{Continuation of the Bishop family}

The following statement shows that transversality always holds for the linearised Cauchy--Riemann operator for maps belonging to the Bishop family. 

\begin{proposition} \emph{(\cite[Theorem 17]{Ho})} \label{prop:openess}
Let $F$ satisfy Equation \ref{eq:holDiscs}. Then there is a smooth family of disjoint embeddings $F_s$, $s \in (-\varepsilon, \varepsilon)$, satisfying Equation \ref{eq:holDiscs}, such that $F_0 = F$. Additionally, any two such families are related by a reparametrisation of the parameter space and a smooth family of Mobius transformations. 
\end{proposition}

\subsubsection{Properties of the Bishop family}

Convexity of $\{0\} \times \SL$ inside of $\R \times \SL$ and an application of the maximum principle yield the following lemma. It will be useful to show that there is no disc bubbling.

\begin{lemma} \emph{(\cite[Lemma 19]{Ho})} \label{lem:winding}
Let $F: (\mathbb{D}^2, \partial \mathbb{D}^2) \to (\R \times \SL, \{0\} \times \Sigma)$ be a $J$--holomorphic map. Then $F(\partial \mathbb{D}^2)$ is transverse to the characteristic foliation of $\Sigma$ and $F(\mathbb{D}^2)$ is transverse to $\{0\} \times \SL$.
\end{lemma}

In order to apply Theorem \ref{thm:removalSingularities} we must have energy bounds, which are provided by the following result.

\begin{proposition} \emph{(\cite[Proposition 27]{Ho}} \label{prop:energyBounds}
There are uniform energy bounds  $0 < C_1 < E(F),E^h(F) < C_2 < \infty$ for every $F$ satisfying Equation \ref{eq:holDiscs} and having
\[ \dist(\im(F), \Gamma_\Sigma) > \varepsilon > 0. \]
\end{proposition}
\begin{proof}
By Stokes' theorem:
\[ E(F) = \sup_{\phi \in \Gamma} \int_{\mathbb{D}^2} F^* d(\phi\alpha) =  \sup_{\phi \in \Gamma} \int_{\partial \mathbb{D}^2} F^* \phi\alpha = \]
\[ \sup_{\phi \in \Gamma}  \phi(0) \int_{\partial \mathbb{D}^2} F^*\alpha = \int_{F(\partial \mathbb{D}^2)} \alpha.  \]
$F(\partial \mathbb{D}^2)$ winds around the critical point exactly once and hence bounds a disc within $\Sigma$. The area of such a disc is always bounded above by a universal constant and is bounded below under the assumption that they have radius at least $\varepsilon$. The claim follows.

A similar estimate holds for $E^h$. 
\end{proof}

\subsection{Proof of Theorem \ref{thm:main}}

Now we tie all the results we have discussed so far.

\begin{lemma} \label{lem:jump1ot}
Let $\SL$ be a leaf of $\SF$ and assume that $(\SL, \xi)$ is an overtwisted contact manifold. Then there is a finite energy plane contained in $\R \times \SL' \in \SF_\mathbb{R}$, with $\SL'$ lying in the closure of $\SL$. 
\end{lemma}
\begin{proof}
Proposition \ref{prop:Bishop} shows that the set of solutions of Equation \ref{eq:holDiscs} is non--empty and Proposition \ref{prop:openess} shows that, up to Moebius transformations, it is an open $1$--dimensional manifold. Denote by $\mathcal{M}$ the component that contains the solutions arising from the elliptic point. 

The boundaries of the maps in $\mathcal{M}$ are pairwise disjoint by Proposition \ref{prop:openess} and they remain transverse to the characteristic foliation by Lemma \ref{lem:winding}. Hence, they define an open submanifold $D$ of the overtwisted disc.

Take a sequence in $\mathcal{M}$ whose distance to the elliptic point is uniformly bounded from below. Then Proposition \ref{prop:ArzelaAscoli} says that either the gradient is unbounded in the family or their limit (by taking a subsequence) is a new solution of Equation \ref{eq:holDiscs}. 

We conclude that if the gradient does not explode, $D$ is both compact and open. Then $D$ should have a tangency with the boundary of the overtwisted disk, which is a contradiction with Lemma \ref{lem:winding}.

Since the gradient explodes, we know by Propositions \ref{prop:planeBubbling} and \ref{prop:discBubbling} that either a plane or a disc bubble appears. In the case of a disc bubble, the standard analysis as in \cite{Ho} shows that bubbles connect, and hence we must have two $J$--holomorphic discs touching at a point and whose winding numbers add up to $1$. This is a contradiction with Lemma \ref{lem:winding}. 

We conclude that necessarily a plane bubble must appear.
\end{proof}

\begin{lemma} \label{lem:jump1pi2}
Let $\SL$ be a leaf of $\SF$ and assume that $\pi_2(\SL) \neq 0$. Then there is a finite energy plane contained in $\R \times \SL \in \SF_\mathbb{R}$, with $\SL'$ lying in the closure of $\SL$. 
\end{lemma}
\begin{proof}
Let us denote by $p_-$ and $p_+$ the two elliptic points of the convex $2$--sphere $\Sigma$ realising a non trivial element of $\pi_2(\SL)$. Proposition \ref{prop:Bishop} gives two different Bishop families starting at each point, which we denote by $\mathcal{M}^-$ and $\mathcal{M}^+$, respectively.

Assume that the gradient is uniformly bounded in the Bishop family $\mathcal{M}^-$. Then $\mathcal{M}^-$ is open and compact, and an application of Proposition \ref{prop:openess} shows that it can be continued until the boundaries of the discs in the family reach $p_+$. Since we know by Proposition \ref{prop:Bishop} that in a neighbourhood of $p_+$ the only curves are those in $\mathcal{M}^+$, both families must be the same. The evaluation map 
\[ \ev: \mathcal{M}^- \times \mathbb{D}^2 \approx [0,1] \times \mathbb{D}^2 \to \SL \]
\[ \ev(F=(a,u), z) = u(z) \]
satisfies $\ev(\partial (\mathbb{M}^- \times \mathbb{D}^2)) = \Sigma$, which contradicts the fact that $\Sigma$ was non--trivial in $\pi_2(\SL)$.

Therefore, the gradient must explode, and since a disc bubble cannot appear, the claim follows.
\end{proof}

We are now ready to prove Theorem \ref{thm:main}.

\begin{proof}[Proof of Theorem \ref{thm:main}]
Let $(\SL, \xi)$ be overtwisted. Proposition \ref{lem:jump1ot} yields a finite energy plane $F: \mathbb{C} \to \R \times \SL$, with $\SL'$ a leaf of $\SF$ contained in the closure of $\SL$. By Lemma \ref{lem:noSphereBubbling} this plane cannot be completed to a sphere. Now an application of Theorem \ref{thm:removalSingularities} shows that there is a closed Reeb orbit in some leaf $\SL''$ lying in the closure of $\SL'$. Since $\SL''$ is in the closure of $\SL$ the claim follows.

Same argument goes through by applying Proposition \ref{lem:jump1pi2} if $\pi_2(\SL) \neq 0$. 
\end{proof}

\begin{remark}
As we have seen, Lemmas \ref{lem:jump1ot} and \ref{lem:jump1pi2} yield a finite energy plane in a leaf that might not be the one containing the overtwisted disc or the convex $2$--sphere. Then, an application of Theorem \ref{thm:removalSingularities} shows that the plane is asymptotic to a trivial cylinder that might live yet in a nother leaf. 

Our example in Subsection \ref{ssec:sharp} shows that at least one of these two phenomena must take place. Is it possible for a ``\textsl{double jump}'' to actually happen?
\end{remark}

\begin{remark}
Let $(M^{2n+1+m}, \SF^{2n+1}, \xi)$ be a contact foliation. Let $\SL$ be a leaf of $\SF$ and let $(\SF, \xi)$ be overtwisted in the sense of \cite{BEM}. More generally, assume that $(\SF, \xi)$ contains a plastikstufe \cite{Nie}.

It is immediate that the Bishop family arising from the plastikstufe can be employed to show that there must be a Reeb orbit, so Theorem \ref{thm:main} also holds true for overtwisted manifolds in all dimensions.
\end{remark}

\section{The non--degenerate case}

In this section we show that under non--degeneracy assumptions none of the \textsl{jumps} between leaves can happen.

\begin{definition}
Let $(M, \SF, \xi)$ be a contact foliation and let $\alpha$ be the defining $1$--form for some extension $\Theta$ of $\xi$. 

A closed orbit of the Reeb vector field associated to $\alpha$ is called \textbf{non--degenerate} if it is isolated among Reeb orbits having the same period and lying in the same leaf of $\SF$.

The form $\alpha$ is called \textbf{non--degenerate} if all the closed orbits of its Reeb vector field are non--degenerate. 
\end{definition}

The statement we want to show is the following. It is a stronger version of the Removal of Singularities (Theorem \ref{thm:removalSingularities}) in the non--degenerate case. 

\begin{theorem} \label{thm:nonDegenerateRemovalSingularities}
Let $(M, \SF, \xi)$ be a contact foliation and let $\alpha$ be the defining $1$--form for some extension $\Theta$ of $\xi$. Assume $\alpha$ is non--degenerate.

Let $F=(a,u): \mathbb{D}^2 \setminus \{0\} \to \R \times \SL \subset \R \times M$ be a $J$--holomorphic curve with $0 < E(F) < \infty$, $\SL$ a leaf of $\SF$. 

Then, either $F$ extends to a $J$--holomorphic map over $\mathbb{D}^2$ or the curves $\gamma_r(s) = u(e^{r+is})$ converge in $C^\infty$ to a closed Reeb orbit $\gamma$ lying in $\SL$. 
\end{theorem}
\begin{proof}
We proceed by contradiction. Assume that $\gamma$, the limit of some $\gamma_{r_i}$, $r_i \to \infty$, is contained in some leaf $\SL' \neq \SL$.

Denote $T = \int_\gamma \alpha$, the period of $\gamma$. By our assumption on $\alpha$, we can find a closed foliation chart $U \subset M$ diffeomorphic to $\mathbb{D}^2 \times \mathbb{S}^1 \times [-1,1]$ around $\gamma$ such that the plaque in $U$ containing $\gamma$ intersects no other orbits of period $T$. Write $h: U \to [-1,1]$ for the height function of the chart: we can assume that $h^{-1}(0)$ is the plaque containing $\gamma$. 

Since the curves $\gamma_{r_i}$ converge in $C^\infty$ to $\gamma$, their images are contained in $U$ for large enough $i$.  Assume, by possibly restricting to a subsequence, that each $\im(\gamma_{r_i})$ lies in a different plaque of $\SF \cap U$. Then, for each $i$, there is a smallest radius $r_i < R_i < r_{i+1}$ such that $\im(\gamma_r)$ is disjoint from the plaque containing $\im(\gamma_{r_i})$, for all $r > R_i$. In particular, $\im(\gamma_{R_i})$ intersects $\partial U$. 

Consider the maps 
\[ F_i: [r_i - R_i, r_{i+1} - R_i] \times \mathbb{S}^1 \to \R \times M \]
\[ F_i(t,s) = (a(e^{t+R_i+is}) - a(e^{R_i}), u(e^{t+R_i+is})) \]
By construction, $F_i(0,0) \in \{0\} \times M$, $F_i(0,s) \cap \{0\} \times (\partial U) \neq \emptyset$, and $\lim_{i \to \infty} h \circ F_i = 0$

By carrying out the bubbling analysis, we can assume that the $F_i$ have bounded gradient. In particular, $r_{i+1} - r_i$ must be uniformly bounded from below by a non--zero constant. Arcel\'a--Ascoli states that the $F_i$ converge in $C^\infty_{loc}$ --maybe after taking a subsequence-- to a map $F_\infty$ with $E^h(F_\infty) = 0$ and therefore lying on top of some Reeb orbit.

By the properties of the $F_i$, $F_\infty$ must have image contained in $\R \times \SL'$ and intersecting  $\R \times (h^{-1}(0) \cap \partial U)$. In particular, $\im(F_\infty)$ is not contained in $\R \times \gamma$. If $\lim_{i \to \infty} R_i-r_i < \infty$, the curves $s \to F_i(r_i-R_i,s)$ would converge to $\gamma$, which is a contradiction. Similarly we deduce that $\lim_{i \to \infty} r_{i+1} - R_i = \infty$. 

Since it has finite energy, $F_\infty: (-\infty, \infty) \times \mathbb{S}^1 \to \R \times \SL'$ must yield a periodic orbit of the Reeb. It must be a closed orbit different from $\gamma$, having period $T$ and intersecting the plaque containing $\gamma$, which is a contradiction. 

We have proved that the limit must lie in $\SL$. It is standard then that the limit does not depend on the sequence chosen $r_i$.
\end{proof}

\begin{remark}
Theorem \ref{thm:nonDegenerateRemovalSingularities} immediately implies that a finite energy plane is asymptotic to a trivial cylinder lying in the same leaf. 

Similarly, it shows that the Bishop family always yields a plane bubble in the original leaf $\SL$: outside of a finite set of points, the Bishop family converges to foliated $J$--holomorphic curve with boundary in the overtwisted disc and possibly many punctures that are asymptotic at $-\infty$ to a number of Reeb orbits necessarily lying in $\SL$.
\end{remark}

\end{document}